\documentclass[12pt,reqno]{article}

\usepackage{amssymb}
\usepackage{graphicx}
\usepackage{amscd}
\usepackage{amsthm}

\usepackage{fullpage}
\usepackage{float}

\usepackage{graphics, amsmath}
\usepackage{amsfonts}
\usepackage{latexsym}
\usepackage{epsf}
\usepackage{amscd}

\begin{document}

\theoremstyle{plain}
\newtheorem{theorem}{Theorem}
\newtheorem{corollary}[theorem]{Corollary}
\newtheorem{lemma}{Lemma}
\newtheorem{proposition}[theorem]{Proposition}

\theoremstyle{definition}
\newtheorem{definition}[theorem]{Definition}
\newtheorem{example}[theorem]{Example}
\newtheorem{conjecture}[theorem]{Conjecture}

\theoremstyle{remark}
\newtheorem{remark}{Remark}

\begin{center}
\vskip 1cm{\LARGE\bf A Generalization of the Eulerian Numbers}
\vskip 1cm
\large
\textbf{Grzegorz Rz\c{a}dkowski,}\footnote{Department of Finance and Risk Management, Warsaw University of Technology, Narbutta 85, 00-999 Warsaw, Poland, e-mail: grzegorz.rzadkowski@pw.edu.pl} \textbf{Ma\l gorzata Urli\'nska}\footnote{Faculty of Mathematics and Computer Science, Warsaw University of Technology, Koszykowa 75, 00-662 Warsaw,  Poland, e-mail: m.urlinska@gmail.com}\\
\end{center}

\newcommand{\Eulerian}[2]{\genfrac{<}{>}{0pt}{}{#1}{#2}}

\begin{abstract}
In the present paper we generalize the Eulerian numbers (also of the second and third orders). The generalization is connected with an autonomous first-order differential equation, solutions of which are used to obtain integral representations of some numbers, including the Bernoulli numbers. 
\end{abstract}


\section{Introduction}
While working on alternating sums of the form $\sum_{k=1}^{m}(-1)^kk^n$, and then more generally on $\sum_{k=1}^{m}j^kk^n$, in 1736 Euler \cite[Ch.\ 7]{E} discovered polynomials and numbers, which now are known as the Eulerian polynomials and the Eulerian numbers. The Eulerian numbers, denoted by $\Eulerian{n}{k}$, are defined by the recurrence formula
\begin{equation}\label{a0}
	\Eulerian{n+1}{k}=(k+1)\Eulerian{n}{k}+(n-k+1)\Eulerian{n}{k-1},
\end{equation}
with boundary conditions $\Eulerian{0}{0}=1$, $\Eulerian{n}{k}=0$ for $k<0$ or $k\ge n$. Many authors use slightly different definition and different notation for Eulerian numbers. In their notation $A(n,k)=\Eulerian{n}{k-1}$.  Some initial terms of the Eulerian numbers ID in OEIS (A008292) are shown in Table~\ref{tab1}.

\begin{table}
\caption{First few Eulerian numbers  $\Eulerian{n}{k}$ }\label{tab1}
\begin{center}
\begin{tabular}{c|ccccccccc}
\hline
$n\backslash k$ & 0 & 1 & 2 & 3 & 4 & 5 &6 &7\\
\hline
1  & 1 & 0 & 0 & 0 & 0 & 0 &0&0\\ 
2  & 1 & 1 & 0 & 0 & 0 & 0 &0&0\\ 
3  & 1 & 4 & 1 & 0 & 0 & 0 &0&0\\ 
4 & 1 & 11 & 11 & 1 & 0 & 0&0 &0\\
5 & 1  & 26 & 66 & 26 & 1 & 0&0 &0\\
6& 1 & 57  & 302 & 302 & 57 & 1 &0&0\\
7 & 1  & 120 & 1191 &  2416&  1191& 120 &1&0\\			
\hline
\end{tabular}
\end{center}
\end{table}

In the 1950s Riordan \cite{R} found an interesting combinatorial property of the Eulerian numbers, which now is often used as their definition. It states that $\Eulerian{n}{k}$ is the number of all permutations of the set $(1,2,\ldots, n)$, with exactly $k$ ascents. The ascent appears in a given permutation if for two consecutive terms of the permutation the first term is smaller than the second one.

The  second-order Eulerian numbers ID (A008517) in OEIS, denoted by $\left < \!\! \Eulerian{n}{k} \!\! \right >$, are defined by the following recurrence formula:
\begin{equation}\label{b0}
	\left < \!\! \Eulerian{n+1}{k} \!\! \right >=(k+1)\left < \!\! \Eulerian{n}{k} \!\! \right >+(2n-k+1)\left < \!\! \Eulerian{n}{k-1} \!\! \right >,
\end{equation}
with the same boundary conditions as for the Eulerian numbers. Table~\ref{tab2} shows the first few of these numbers.

\begin{table}
\caption{First few second-order Eulerian numbers $\left < \!\! \Eulerian{n}{k} \!\! \right >$  }\label{tab2}
\begin{center}
\begin{tabular}{c|ccccccccc}
\hline
$n\backslash k$ & 0 & 1 & 2 & 3 & 4 & 5 &6 &7\\
\hline
1  & 1 & 0 & 0 & 0 & 0 & 0 &0&0\\ 
2  & 1 & 2 & 0 & 0 & 0 & 0 &0&0\\ 
3  & 1 & 8 & 6 & 0 & 0 & 0 &0&0\\ 
4 & 1 & 22 & 58 & 24 & 0 & 0&0 &0\\
5 & 1  & 52 & 328 & 444 & 120 & 0&0 &0\\
6& 1 & 114  & 1452 & 4400 & 3708 & 720 &0&0\\
7 & 1  & 240 & 5610 &  32120&  58140& 33984 &5040&0\\		
\hline
\end{tabular}
\end{center}
\end{table}

Eulerian numbers have been well described in books by Comtet \cite{C}, Graham et al. \cite{GKP} and Riordan \cite{R} or by Foata in his review article \cite{F}.   Some applications and generalizations of Eulerian numbers are given in many articles. For example, Carlitz \cite{Ca1, Ca2} introduced q-Eulerian numbers, Lehmer \cite{L} defined a generalization based on the iteration of a differential operator, Butzer and Hauss \cite{BH} described Eulerian numbers with fractional order parameters. Barbero et al. \cite{BSV} essentially solved the problem 6.94, stated in Graham et al. \cite{GKP}, about a solution of the general linear recurrence equation. Chung et al. \cite{CGK} gave three proofs of an identity for Eulerian numbers, He \cite{He} demonstrated their application to B-splines and Banaian \cite{Ba} used them in the theory of juggling. The last application of the Eulerian numbers is especially interesting for the first author because his son Stanis{\l}aw is able to juggle with seven balls.

In Sec.~\ref{sec2} we define a sequence $(G(n, k))$, which is a natural generalization of Eulerian numbers. 
In Sec.~\ref{sec3}  we show that the generalized Eulerian numbers are associated with an autonomous differential equation. Then in Sec.~\ref{sec4} we use, on some examples, solutions of the equation to get integral representations for some numbers, including the Bernoulli numbers.

\section{Definition and properties}\label{sec2}
Let us define a sequence $G(n,k)$ by the recurrence formula
\begin{equation}\label{a1}
	G(n+1,k)=(nw_1-n+k+1)G(n,k)+(nw_2-k+1)G(n,k-1),
	\end{equation}
for integer numbers $n\ge 0,k$ and real parameters $w_1, w_2$, with boundary conditions $G(0,0)=1$, $G(n,k)=0$ for $k<0$ or $k\ge n$.
	
For $w_1=w_2=1$ formula (\ref{a1}) gives Eulerian numbers ID in OEIS (A008292), for $w_1=1,w_2=2$, the second-order Eulerian numbers ID in OEIS (A008517) and if $w_1=1,w_2=3$,  the third-order Eulerian numbers ID in OEIS (A219512). Thus sequence (\ref{a1}) is a natural generalization of Eulerian numbers.

\begin{lemma}\label{lem1}
The sum of the $n$th  row ($n=1,2,3\ldots$) in the array of coefficients $G(n,k)$, for $k=0,\;1,\;2,\ldots,\; n-1$ is a polynomial of $(w_1+w_2)$ of order $n-1$,  and the following formula holds:
		\begin{equation}\label{b1}
		\sum\limits_{k=0}^{n-1}G(n,k)=\prod\limits_{m=0}^{n-1}(m(w_1+w_2)-m+1).
	\end{equation}
\end{lemma}
\begin{proof}
It is easy to check that adding by sides, for a given $n$, all $(n+1)$ identities (\ref{a1}) respectively, for $k=0,1,2,\ldots,n$ we arrive at the following recurrence formula:
\begin{equation}\label{c1}
	\sum\limits_{k=0}^{n}G(n+1,k)=(n(w_1+w_2)-n+1)\sum\limits_{k=0}^{n-1}G(n,k).
\end{equation}

Since $G(1,0)=1$ and $G(2,0)+G(2,1)=w_1+w_2$ then, by induction, from identity (\ref{c1}) we get formula (\ref{b1}).
\end{proof}

Lemma~\ref{lem1} is a particular case of a similar result, obtained by Neuwirth \cite{N} for a more general sequence
	\[G(n+1,k)=(\alpha n+\beta k+\lambda)G(n,k)+(\alpha' n+\beta' k+\lambda')G(n,k-1), \quad \beta+\beta'=0.
\]
Spivey \cite{S} also gives some remarks connected with this result.

Let us denote by $P_{n-1}(x)$ the polynomial of a variable $x$ of order $(n-1)$, which we get by substituting $x=w_1+w_2$ on the right hand side of identity (\ref{b1}), i.e.,
\begin{equation}\label{d1}
	P_{n-1}(x)=\prod\limits_{m=0}^{n-1}(mx-m+1).
\end{equation}

Expanding $P_{n-1}(x)$ for $n=2,3\ldots$ into powers of $x$ and denoting a coefficient  of $x^k$ by $M(n,k)$, we obtain
	  \[P_{n-1}(x)=M(n,1)x+M(n,2)x^2+\cdots +M(n,n-1)x^{n-1}.
	\]
	
It is easy to see that the coefficients ($M(n,k)$), by the formula for generating polynomial (\ref{d1}), fulfill the following recurrence formula (with boundary conditions $M(1,0)=1$, $M(n,0)=0$ for $n\ge 2$, $M(n,k)=0$ for $k\ge n$):
	\[M(n+1,k)=(1-n)M(n,k)+nM(n,k-1).
\]

Table~\ref{tab3} shows the first few elements of the sequence ($M(n,k)$). The sequence ($|M(n,k)|$) appears in OEIS with ID (A059364), however without any information about a generating polynomial of the $n$th row of the array. The polynomial, in this case, has the following form:
	\[|M(n,1)|x+|M(n,2)|x^2+\cdots +|M(n,n-1)|x^{n-1}=\prod\limits_{m=1}^{n-1}(mx+m-1).
\]

\begin{table}
\caption{First few numbers of the sequence ($M(n,k)$) }\label{tab3}
\begin{center}
\begin{tabular}{c|cccccccc}
\hline
$n\backslash k$ & 0 & 1 & 2 & 3 & 4 & 5 &6 &7\\
\hline
1  & 1 & 0 & 0 & 0 & 0 & 0 &0&0\\ 
2  & 0 & 1 & 0 & 0 & 0 & 0 &0&0\\
3 & 0 & -1 & 2 & 0 & 0 & 0 &0&0\\
4 & 0  & 2& -7& 6 & 0 & 0 &0&0\\
5& 0 & -6 & 29 & -46 & 24 & 0 &0&0\\
6 & 0  & 24 & -146 & 329 & -326 & 120 &0 &0\\
7 & 0 & -120 & 874 & -2521 & 3604 &  -2556&720&0\\									
\hline
\end{tabular}
\end{center}
\end{table}

\section{Main results}\label{sec3}
Let $u=u(z)$ be a holomorphic function defined in a domain $D \subset \mathbb{C}$ which fulfills the  following autonomous first-order differential equation with constant coefficients:
\begin{equation}\label{a2}
	u'(z)=r(u-a)^{w_1}(u-b)^{w_2},
\end{equation}
where $r,a,b$ are real or complex numbers, $r\neq 0, \; a \neq b$ and $w_1, w_2$ are real numbers. We understand the powers $(u-a)^{w_1} $ and $(u-b)^{w_2}$  as $e^{w_1\log(u-a)}$, $e^{w_2\log(u-b)}$ for a branch of the logarithm function.

\begin{theorem}\label{th1}
If a function $u(z)$ satisfies equation (\ref{a2}),
then the $n$th derivative of $u(z)$ equals
\begin{equation}\label{b2}
u^{(n)}(z) = r^{n}\sum\limits_{k=0}^{n-1}
G(n,k)(u-a)^{nw_1-n+k+1}(u-b)^{nw_2-k}
\end{equation}
where $n=2,3,\ldots $.
\end{theorem} 
\begin{proof} 
We will proceed by induction. For $n=1$ formula (\ref{b2}) becomes definition (\ref{a2}), therefore is true. Let us assume that for a positive integer $n$ formula (\ref{b2}) holds. By the chain rule and recurrence formula (\ref{a1})  we get 
\begin{align*}
 u^{(n+1)}(z)=& r^{n}\frac{d}{dz}\sum\limits_{k=0}^{n-1} G(n,k)(u-a)^{nw_1-n+k+1}(u-b)^{nw_2-k}\nonumber \\
=& r^{n}\sum\limits_{k=0}^{n-1}G(n,k)\left[(nw_1-n+k+1)(u-a)^{nw_1-n+k}(u-b)^{nw_2-k} \right.\\
&\left.+(nw_2-k)(u-a)^{nw_1-n+k+1}(u-b)^{nw_2-k-1}\right]r(u-a)^{w_1}(u-b)^{w_2}\\
=&r^{n+1}\left[G(n,0)(nw_1-n+1)(u-a)^{(n+1)w_1-n}(u-b)^{(n+1)w_2}\right.\\
&+\sum\limits_{k=1}^{n-1}((nw_1-n+k+1)G(n,k)+(nw_2-k+1)G(n,k-1))(u-a)^{(n+1)w_1-n+k}\\
&\left. \times(u-b)^{(n+1)w_2-k}+G(n,n-1)(nw_2-n+1)(u-a)^{(n+1)w_1}(u-b)^{(n+1)w_2-n}\right]\\
=&r^{n+1}\sum\limits_{k=0}^{n}
G(n+1,k)(u-a)^{(n+1)w_1-(n+1)+k+1}(u-b)^{(n+1)w_2-k}
\end{align*}
which ends the proof.
\end{proof}
\begin{remark}
Formula (\ref{b2}) may also be formulated in a real sense. Let us rewrite equation (\ref{a2}) as
\begin{equation}\label{c2}
	v'(z)=s(v-a)^{w_1}(b-v)^{w_2},
\end{equation}
where $s,a,b, w_1, w_2$ are real numbers, $s\neq 0$ and $a  <b$. We are looking for a real solution $v(z)$ of (\ref{c2}), which fulfills the condition $a\le v(z) \le b$, where $z\in D \subset \Re$ ($D$ interval). Then if $v(z)$ be such a solution then its $n$th derivative is given by the following formula:
\begin{equation}\label{d2}
v^{(n)}(z) = s^{n}\sum\limits_{k=0}^{n-1}(-1)^k
G(n,k)(v-a)^{nw_1-n+k+1}(b-v)^{nw_2-k}.
\end{equation}
The proof of (\ref{d2}) is analogous to that of Theorem~\ref{th1}.
\end{remark}

Let us denote by $G_{n}(u)$ and $H_{n}(u)$ the sums on the right hand side of formulas (\ref{b2}) and  (\ref{d2}) respectively, i.e.,
\begin{equation}\label{h2}
		G_{n}(u)=\sum\limits_{k=0}^{n-1}
		G(n,k)(u-a)^{nw_1-n+k+1}(u-b)^{nw_2-k},
\end{equation}
\begin{equation}\label{i2}
	H_{n}(v)=\sum\limits_{k=0}^{n-1}(-1)^k
	G(n,k)(v-a)^{nw_1-n+k+1}(b-v)^{nw_2-k}.
\end{equation}

When $u(z)$ is a solution of (\ref{a2}), given in an explicit form, then we are able to compute the exponential generating function for the functions $(G_{n}(u))_{n\ge 1}$ in the following form:
	\[g(u,w)=u+rG_{1}(u)w+r^2G_{2}(u)\frac{w^2}{2!}+r^3G_{3}(u)\frac{w^3}{3!}+\cdots.
\]
In order to do it let us observe that $g(u(z),w)$ is the Taylor expansion of $u(z)$ at the point $z$. Therefore we obtain
\begin{align}
g(u(z),w)=& u(z)+rG_{1}(u(z))w+r^2G_{2}(u(z))\frac{w^2}{2!}+r^3G_{3}(u(z))\frac{w^3}{3!}+\cdots \nonumber \\
=&u(z) +u'(z)w+u''(z)\frac{w^2}{2!}+u'''(z)\frac{w^3}{3!}+\cdots =u(z+w).\label{e2}
\end{align}
Assuming moreover that $u(z)$ is invertible on the set $D$ we obtain from (\ref{e2}) the following formula for the generating function $g(u,w)$:
\begin{equation}\label{f2}
g(u,w)=u(z(u)+w).
	\end{equation}
 
 Similarly we can express the exponential generating function for functions $(H_{n}(v))_{n\ge 1}$ in the form
	\[h(v,w)=v+sH_{1}(v)w+s^2H_{2}(v)\frac{w^2}{2!}+s^3H_{3}(v)\frac{w^3}{3!}+\cdots,
\]
associated with a solution $v(z)$ of equation (\ref{c2}). With the same assumptions as before we have
\begin{equation}\label{g2}
h(v,w)=v(z(v)+w).
	\end{equation}

\section{Examples}\label{sec4}
\textbf{Example 1.} Putting $w_1=w_2=1$ into (\ref{a1}) we obtain recurrence (\ref{a0}) for theEulerian numbers. Equation (\ref{a2}) in this case is the Riccati differential equation with constant coefficients  
\begin{equation}\label{a4}
	u'(z)=r(u-a)(u-b),
\end{equation}
and formula (\ref{b2}) takes the following form:
\begin{equation}\label{b4}
u^{(n)}(z) = r^{n}\sum\limits_{k=0}^{n-1} \Eulerian{n}{k}
(u-a)^{k+1}(u-b)^{n-k},
\end{equation}
where $n=2,3,\ldots $. 

Formula (\ref{b4}) was discussed during the Conference ICNAAM 2006 in Greece and it appeared, with an inductive proof, in  paper \cite{Rz} (see also \cite{Rz1}). Independently Franssens \cite{F} considered and proved formula (\ref{b4}), giving a proof based on generating functions. Hoffman \cite{Ho} introduced so called derivative polynomials and stated similar problems for the first time.

In paper \cite[Th.\ 2.2, p.\ 124]{Rz2} we have proved (taking in (\ref{b4}): $r=1$, $a=0$ and $b=1$)
\begin{equation}\label{c4}
	\int_{0}^{1}\ \sum\limits_{k=0}^{n-1}\Eulerian{n}{k} 
u^{k+1}(u-1)^{n-k}du=-B_{n+1},
\end{equation}
where $B_n$ is the $n$th Bernoulli number. Bernoulli numbers are well described in books by Graham et al. \cite{GKP} or Duren \cite{D}. Let us mention that for such parameters one of the solutions of the Riccati equation (\ref{a4}) is $u(z)=1/(1+e^z)$.

From equation (\ref{c4}) we can obtain at least two remarkable formulas. The first one we get by substituting in (\ref{c4}) $u=u(z)=1/(1+e^z)$, which yields
\begin{equation}\label{d4}
	\int_{-\infty}^{\infty}\left(\frac{1}{1+e^z}\right)^{(n)}\left(\frac{1}{1+e^z}\right)'dz=B_{n+1}.
\end{equation}

Since 
	\[\left(\frac{1}{1+e^z}\right)'=\frac{-e^z}{(1+e^z)^2}=-\frac{1}{4\cosh^2(\frac{z}{2})},
\]
then, assuming that $n+1$  is an even number, say $n+1=2m$, we substitute in (\ref{d4}) $x=z/2$ and integrate $m-1$ times by parts. In this way we arrive at the Grosset-Veselov formula
					\[\frac{(-1)^{m-1}}{2^{2m+1}}\int_{-\infty}^{\infty}\left(\frac{d^{m-1}}{dx^{m-1}}\frac{1}{\cosh^2(x)}\right)^2 dx=B_{2m}.
		\]
Grosset and Veselov \cite{GV} obtained it while examining soliton solutions of the KdV equation. 

The second formula we get, directly integrating the left hand side of equation (\ref{c4}). Since 
	\[\int_{0}^{1}u^{k+1}(u-1)^{n-k}du= (-1)^{n-k}\frac{(k+1)!(n-k)!}{(n+2)!}=
	\frac{(-1)^{n-k}}{n+2}\frac{1}{{n+1 \choose k+1}},
\]
then we can rewrite (\ref{c4}) in the following form:
\begin{equation}\label{e4}
	\sum\limits_{k=0}^{n-1}(-1)^{n-k+1} \frac{\Eulerian{n}{k}}{{n+1 \choose k+1}}=(n+2)B_{n+1},
\end{equation}
 valid for $n=1,2,3,\ldots$. 

\textbf{Example 2.} Putting $w_1=w_2=\frac12$ into (\ref{a1}) we obtain the following recurrence formula:
\begin{equation}\label{f4}
	G(n+1,k)=\bigl(-\frac{n}{2}+k+1\bigr)G(n,k)+\bigl(\frac{n}{2}-k+1\bigr)G(n,k-1).
\end{equation}
From (\ref{f4}) it follows that if $n$ is an odd positive integer then $G(n,k)$ differs from zero only for $k=\frac{n-1}{2}$ with $G(n,\frac{n-1}{2})=1$. When $n\ge 2$ is an even integer then only $G(n,\frac{n}{2})=G(n,\frac{n}{2}-1)=\frac12$ are different from zero (see Table~\ref{tab4}).

\begin{table}
\caption{First few numbers of the sequence ($G(n,k)$) for $w_1=w_2=\frac12$}\label{tab4}
\begin{center}
\begin{tabular}{c|ccccccc}
\hline
$n\backslash k$ & 0 & 1 & 2 & 3 & 4 & 5  \\
\hline
1  & 1 & 0 & 0 & 0 & 0 & 0 \\ 
2  & $\frac12 $ & $\frac12 $ & 0 & 0 & 0 & 0 \\
3 & 0 & 1 & 0 & 0 & 0 & 0 \\
4 & 0  & $\frac12 $ & $\frac12 $ & 0 & 0 & 0 \\
5& 0 & 0  & 1 & 0 & 0 & 0 \\
6 & 0  & 0 & $\frac12 $ & $\frac12 $ & 0 & 0  \\
7 & 0 & 0 & 0 & 1 & 0 & 0 \\
8 & 0  & 0 & 0 & $\frac12 $ & $\frac12 $ & 0   \\
\hline
\end{tabular}
\end{center}
\end{table}

Let us consider a differential equation of type (\ref{c2}) associated with the sequence (\ref{f4}), say for $s=1,\; a=-1,\; b=1,\; w_1=w_2=\frac12$
\begin{equation}\label{g4}
	v'(z)=(v+1)^{\frac12}(1-v)^{\frac12}=\sqrt{1-v^2}.
\end{equation}

One of the solutions of (\ref{g4}), for the initial condition $v(0)=0$, is $v(z)=\sin z$ for $z\in [-\pi/2,\pi/2] $. Using formulas (\ref{i2}) and (\ref{d2}), or directly from derivatives of $v(z)=\sin z$, it follows that $H_1(v)=\sqrt{1-v^2}$, $H_2(v)=-v$, $H_3(v)=-\sqrt{1-v^2}$, $H_4(v)=v$ and 
$H_n(v)=H_{n-4}(v)$ for $n\ge5$. 

The generating function  (\ref{g2}) is easy to find in the present case. Since
	\[h(v(z),w)=v(z+w)=\sin(z+w)=\sin z\cos w +\cos z \sin w, 
\]
then 
	\[h(v,w)=u\cos w+\sqrt{1-v^2}\sin w.
\]

Formulas corresponding to (\ref{c4}) and (\ref{d4})  are trivial in this case. Since 
$\int_{-1}^{1}\sqrt{1-v^2}dv=\pi/2$ we have
	\[\int_{-1}^{1}h(v,w)dv=\frac{\pi}{2}\sin w= \frac{\pi}{2}\bigl(w-\frac{w^3}{3!}
	+\frac{w^5}{5!}-\frac{w^7}{7!}+\cdots\bigr),
\]
and then
\begin{displaymath}
 \int_{-1}^{1}H_n(v)dv= \begin{cases}
0, & \text{if $n$ is even;}\\
(-1)^{\frac{n-1}{2}}\frac{\pi}{2}, & \text{if $n$ is odd.}
\end{cases}
\end{displaymath}

Therefore the formula corresponding to (\ref{d4}) is
\begin{displaymath}
 \int_{-\frac{\pi}{2}}^{\frac{\pi}{2}}(\sin z)^{(n)}(\sin z)'dz=
\int_{-\frac{\pi}{2}}^{\frac{\pi}{2}}(\sin z)^{(n)}\cos zdz=
 \begin{cases}
0, & \text{if $n$ is even;}\\
(-1)^{\frac{n-1}{2}}\frac{\pi}{2}, & \text{if $n$ is odd.}
\end{cases}
\end{displaymath}

\textbf{Example 3.} Let us put $w_1=\frac12, w_2=1$ into (\ref{a1}). We have
 \begin{equation}\label{h4}
	G(n+1,k)=\bigl(-\frac{n}{2}+k+1\bigr)G(n,k)+(n-k+1)G(n,k-1).
\end{equation}

Some initial terms of the sequence are shown in Table~\ref{tab5}. 
\begin{table}
\caption{First few numbers of the sequence ($G(n,k)$) for $w_1=\frac12, w_2=1$}\label{tab5}
\begin{center}
\begin{tabular}{c|ccccccccc}
\hline
$n\backslash k$ & 0 & 1 & 2 & 3 & 4 & 5 &6 &7\\
\hline
1  & 1 & 0 & 0 & 0 & 0 & 0 &0&0\\ 
2  & $\frac12 $ & 1 & 0 & 0 & 0 & 0 &0&0\\ 
3  & 0 & 2 & 1 & 0 & 0 & 0 &0&0\\ 
4 & 0 & 1 & $\frac{11}{2}$ & 1 & 0 & 0&0 &0\\
5 & 0  & 0 & $\frac{17}{2} $ & 13 & 1 & 0&0 &0\\
6& 0 & 0  & $\frac{17}{4} $ & 45 & $\frac{57}{2} $ & 1 &0&0\\
7 & 0  & 0 & 0 &  62&  192& 60 &1&0\\		
8 & 0 & 0 & 0 & 31 & 536 & 726 &$\frac{247}{2} $&1\\			
\hline
\end{tabular}
\end{center}
\end{table}

The numerical evidence shows that every odd row of the above sequence corresponds to a certain row of the sequence ID (A160468) in OEIS. 

Let us consider also a differential equation of type (\ref{a2}) associated with sequence 
(\ref{h4}), say for parameters $r=1,\; a=0,\; b=1, \; w_1=\frac12, \; w_2=1$
\begin{equation}\label{i4}
	u'(z)=\sqrt{u}(u-1).
\end{equation}

One of the solutions of the equation (\ref{i4}) is 
\begin{equation}\label{j4}
	u(z)=\left(\frac{1-e^z}{1+e^z}\right)^2,
\end{equation}
which we will consider only in the interval $(-\infty, 0]$, where $u(z)$ is a monotonic (decreasing) function. The inverse function of the function (\ref{j4}) is 
\begin{equation}\label{k4}
	z(u)=\log \frac{1-\sqrt{u}}{1+\sqrt{u}},
\end{equation}
defined in the interval $[0,1)$. 

Now we will find the generating function (\ref{f2}). Since
	\[g(u(z),w)=u(z+w)=\left(\frac{1-e^{z+w}}{1+e^{z+w}}\right)^2,
\]
then using (\ref{k4}) we have
\begin{equation}\label{m4}
	g(u,w)=\left(\frac{1-e^{w}\frac{1-\sqrt{u}}{1+\sqrt{u}}}{1+e^{w}\frac{1-\sqrt{u}}{1+\sqrt{u}}}\right)^2=
	\left(\frac{1+\sqrt{u}-e^{w}(1-\sqrt{u})}{1+\sqrt{u}+e^{w}(1-\sqrt{u})}\right)^2.
\end{equation}

Let us denote by $f(w)$ the integral of the generating function (\ref{m4})
\begin{align*}
	f(w)=&\int_{0}^{1}g(u,w)du=\int_{0}^{1}\left(\frac{1+\sqrt{u}-e^{w}(1-\sqrt{u})}{1+\sqrt{u}+e^{w}(1-\sqrt{u})}\right)^2du \\
	=&\frac{16(e^{2w}+4e^w+1)e^w\log(\frac{2}{e^w+1})+(e^w-1)(e^{3w}+33e^{2w}+15e^w-1)}{(e^w-1)^4}.
\end{align*}

Thus, in this case, the formula corresponding to (\ref{c4}) is
	\[\int_{0}^{1}G_n(u)du=f_n,
\]
and this corresponding to (\ref{d4}) is as follows: 
	\[\int_{-\infty}^{0}(u(z))^{(n)}u'(z)dz=
	\int_{-\infty}^{0}\frac{d^n}{dz^n}\left(\frac{1-e^z}{1+e^z}\right)^2\cdot 
	\frac{d}{dz}\left(\frac{1-e^z}{1+e^z}\right)^2 dz=-f_n,
\]
where $f_n$ is the coefficient of $w^n/n!$ in the Taylor expansion of the function $f(w)$
	\[f(w)=\frac12-\frac{4}{15}w+\frac{4}{21}\frac{w^3}{3!}+\frac18\frac{w^4}{4!}
	-\frac{4}{15}\frac{w^5}{5!}-\frac12 \frac{w^6}{6!}+\frac{20}{33}\frac{w^7}{7!}
	+\frac{21}{8}\frac{w^8}{8!}-\frac{2764}{1365}\frac{w^9}{9!}+O(w^{10}).
	\]

\textbf{Example 4.} By substituting in (\ref{a1}) $w_1=1,\ w_2=2$ we get the recurrence (\ref{b0}) for  
 the second-order Eulerian numbers $\displaystyle G(n,k)=\left < \!\! \Eulerian{n}{k} \!\! \right >$.

The following differential equation: 
\begin{equation}\label{p4}
u'(z)=u(u-1)^2
\end{equation}
associated with (\ref{b0}), unfortunately cannot be explicitly solved. Denoting by $u=L(z)$ the solution of (\ref{p4}) with the initial condition $L(0)=\frac12$, we  can only find  its inverse function 
	\[z(u)=L^{-1}(u)=\log \frac{u}{1-u} +\frac{u}{1-u} -1, \qquad u\in (0,1).
\]

The numerical evidence shows (as checked for $n\le 9$) that 
\begin{equation}\label{q4}
	\int_{0}^{1}G_n(u)du=\int_{0}^{1}\sum\limits_{k=0}^{n-1}
\left < \!\! \Eulerian{n}{k} \!\! \right> u^{k+1}(u-1)^{2n-k}du=\frac{B_{n+1}}{n+1},
\end{equation}
where $B_n$ is the $n$th Bernoulli number.

If the hypothesis (\ref{q4}) were true, it would lead to the following formula:
\begin{align*}
	&\int_{0}^{1}g(u,w)du=\int_{0}^{1}L\bigl(\log \frac{u}{1-u} +\frac{u}{1-u} -1+w\bigr)du \\
	&= \frac12+\frac{B_2}{2}w+\frac{B_3}{3}\frac{w^2}{2!}+\frac{B_4}{4}\frac{w^3}{3!}+\cdots=
	\frac{1}{w}\bigl(-1+w+1+B_1w+B_2\frac{w^2}{2!}+B_3\frac{w^3}{3!}+\cdots \bigr)\\
	&=\frac{1}{w}\bigl(-1+w+\frac{w}{e^w-1}\bigr)=-\frac{1}{w}+\frac{e^w}{e^w-1},
\end{align*}
where we have employed the generating function for Bernoulli numbers, often used as their definition
	\[\frac{w}{e^w-1}=1+B_1w+B_2\frac{w^2}{2!}+B_3\frac{w^3}{3!}+\cdots .
\]

Moreover, since
	\[\int_{0}^{1}u^{k+1}(u-1)^{2n-k}du= (-1)^{k}\frac{(k+1)!(2n-k)!}{(2n+2)!}=
	\frac{(-1)^{k}}{2(n+1)}\frac{1}{{2n+1 \choose k+1}},
\]
then from (\ref{q4}) we would also get
	\[\sum\limits_{k=0}^{n-1}(-1)^k\frac{\left < \!\! \Eulerian{n}{k} \!\! \right>}{{2n+1 \choose k+1}}=2B_{n+1}.
\]

The Reader may find an interesting, but unfinished, discussion on a similar topic on the web page MathOverflow \cite{Ma}.

\end{document}